\documentclass[12pt]{amsart}

\newtheorem{theorem}{Theorem}[section]
\newtheorem{lemma}[theorem]{Lemma}
\newtheorem{proposition}[theorem]{Proposition}

\newtheorem{remark}[theorem]{Remark}

\numberwithin{equation}{section}

\begin{document}
\baselineskip=15pt

\title{Universal vector bundle over the reals}

\author[I. Biswas]{Indranil Biswas}

\address{School of Mathematics, Tata Institute of Fundamental Research,
Homi Bhabha Road, Bombay 400005, India}

\email{indranil@math.tifr.res.in}

\author[J. Hurtubise]{Jacques Hurtubise}

\address{Department of Mathematics, McGill University, Burnside
Hall, 805 Sherbrooke St. W., Montreal, Qu\'e. H3A 2K6, Canada}

\email{jacques.hurtubise@mcgill.ca}



\date{}

\begin{abstract}
Let $X_{\mathbb R}$ be a geometrically irreducible smooth projective
curve, defined over $\mathbb R$, such that $X_{\mathbb R}$ does not
have any real points. Let $X\,=\,X_{\mathbb R}\times_{\mathbb R}
\mathbb C$ be the complex curve. We show that
there is a universal real algebraic line
bundle over $X_{\mathbb R}\times \text{Pic}^d(X_{\mathbb R})$ if and
only if the Euler characteristic
$\chi(L)$ is odd for $L\,\in\, \text{Pic}^d(X_{\mathbb R})$.
There is a universal quaternionic algebraic line
bundle over $X\times \text{Pic}^d(X)$ if and only if the
degree $d$ is odd. (Quaternionic algebraic vector bundles are defined only
on a complexification.)

Take integers $r$ and $d$ such that $r\, \geq\, 2$, and $d$ is coprime
to $r$. Let ${\mathcal M}_{X_{\mathbb R}}(r,d)$
(respectively, ${\mathcal M}_{X}(r,d)$) be the moduli space of 
stable vector
bundles over $X_{\mathbb R}$ (respectively, $X$) of rank $r$ and degree 
$d$. We prove that there is
a universal real algebraic vector bundle over $X_{\mathbb R}\times 
{\mathcal M}_{X_{\mathbb R}}(r,d)$
if and only if $\chi(E)$ is odd for $E\,\in\, {\mathcal 
M}_{X_{\mathbb R}}(r,d)$. There is
a universal quaternionic vector bundle over $X\times {\mathcal M}_X(r,d)$
if and only if the degree $d$ is odd.

The cases where $X_{\mathbb R}$ is geometrically reducible or
$X_{\mathbb R}$ has real points are also investigated. 
\end{abstract}

\maketitle

\section{Introduction}\label{sec1}

Let $Y$ be a compact connected Riemann surface of genus $g$.
For any integers $r\, \geq\,2$ and $d$, let ${\mathcal M}_Y(r,d)$
be the moduli space of stable vector bundles over $Y$ rank $r$ and degree
$d$. A Poincar\'e bundle (also called a universal bundle) is an 
algebraic vector
bundle ${\mathcal E}\, \longrightarrow\, Y\times {\mathcal M}_Y(r,d)$ with the
property that for each point $z\, \in\, {\mathcal M}_Y(r,d)$, the vector
bundle ${\mathcal E}\vert_{Y\times \{z\}}$ on $Y$ is in the isomorphism class of
vector bundles defined by $z$. Similarly, a Poincar\'e
line bundle ${\mathcal L}\,\longrightarrow\, Y\times \text{Pic}^d(Y)$ is an
algebraic line bundle with the
property that for each point $z\, \in\, \text{Pic}^d(Y)$, the restriction
${\mathcal L}\vert_{Y\times \{z\}}$ is in the isomorphism class of
line bundles defined by $z$. 
Poincar\'e vector bundles are very useful objects;
for example, they are used in computing the cohomology of ${\mathcal M}_Y(r,d)$
\cite{MN}, \cite{JK}.

There is a Poincar\'e
line bundle on $Y\times \text{Pic}^d(Y)$ for each $d$.
If $g\,\geq\, 2$, it is known that there is a universal vector bundle over
$Y\times {\mathcal M}_Y(r,d)$ if and only if $d$ is coprime to $r$. A similar
statement holds for the moduli space of stable vector bundles
over $Y$ with fixed
determinant. Our aim here is to address similar questions for curves
defined over the real numbers.

Let $X_{\mathbb R}$ be an irreducible smooth projective curve
defined over $\mathbb R$. If $X_{\mathbb R}$ is geometrically reducible,
or if $X_{\mathbb R}$ has a real point, then it is straight--forward
to answer the question of existence of a universal bundle (see
Section \ref{s2.3} for the details). So assume that
\begin{enumerate}
\item $X_{\mathbb R}$ is geometrically irreducible, and

\item $X_{\mathbb R}$ does not have any real points.
\end{enumerate}

Let $X\, :=\, X_{\mathbb R}\times_{\mathbb R}
\mathbb C$ be the complex curve obtained by base change. The
conjugation of $\mathbb C$ gives an anti--holomorphic involution
$\sigma$ of $X$. Real algebraic vector bundles over
$X_{\mathbb R}$ are the same as complex algebraic vector bundles
$E\,\longrightarrow\, X$ together with an algebraic isomorphism
$\sigma_E\, :\, E\, \longrightarrow\,\sigma^*\overline{E}$ with
$\sigma_E\circ \sigma^*(\overline\sigma_E)\,=\,\text{Id}_E$. Using the
canonical $C^\infty$ isomorphism of $E$ with $\overline{E}$, an
isomorphism of $E$ with $\sigma^*\overline{E}$ is given by an
anti--holomorphic automorphism of the total space of $E$
lifting $\sigma$ which is
conjugate--linear on the fibers. The above condition
$\sigma_E\circ \sigma^*(\overline\sigma_E)\,=\,\text{Id}_E$ is
equivalent to the condition that this automorphism is an involution.

A quaternionic
vector bundle over $X$ is defined to be an algebraic vector bundle 
$E\,\longrightarrow
\,X$ equipped with an algebraic isomorphism
$\sigma_E\, :\, E\, \longrightarrow\,\sigma^*\overline{E}$ such that
$\sigma_E\circ\sigma^*(\overline{\sigma}_E)\,=\,-\text{Id}_E$.
Quaternionic vector bundles are defined only on the complexification
$X$ and not on $X_{\mathbb R}$. Since the automorphism $-\text{Id}_E$ acts
trivially on the projective bundle ${\mathbb P}(E)$ associated to $E$, a
quaternionic vector bundle on $X$ defines a real projective
bundle over $X_{\mathbb R}$.

These descriptions of
real and quaternionic vector bundles hold for any variety defined
over $\mathbb R$. Real algebraic universal vector bundles for $X_{\mathbb 
R}$ will be identified with the universal vector bundles for $X$
equipped with a lift of the anti--holomorphic involution
of the base; they will be called real universal vector bundles.

For a vector bundle $E$ on $X$ of rank $r$ degree $d$, define 
$\chi(E)\,:=\,
\dim H^0(X,\, E)- \dim H^1(X,\, E)$. By Riemann--Roch, this is expressed 
as $d- r(\text{genus}(X)-1)$.

It turns out that
contrary to the complex case, there are obstructions for the
existence of a real or quaternionic universal line bundle over
$X\times {\rm Pic}^d(X)$. More precisely, we prove the following:

\begin{theorem}\label{thm1}
There is a real universal line bundle over $X\times {\rm Pic}^d(X)$
if and only if $\chi(L)$ is odd for $L\, \in\, {\rm Pic}^d(X)$.

There is a quaternionic universal line bundle over
$X\times {\rm Pic}^d(X)$ if and only if $d$ is odd.
\end{theorem}

See Propositions \ref{propexistence}, \ref{prop1} and
\ref{prop2} for a proof of Theorem \ref{thm1}.

Take any integer $r\, \geq\, 2$. Let $d$ be any integer
coprime to $r$ (as mentioned above, there is no universal
vector bundle, even over the complex numbers, if
$\text{g.c.d}(r,d)\,\not=\,1$). Let ${\mathcal M}_X(r,d)$
be the moduli space of stable vector bundles over $X$ of rank
$r$ and degree $d$.

\begin{theorem}\label{thm2}
There is a real universal vector bundle over
$X\times {\mathcal M}_X(r,d)$ if and only if
$\chi(E)$ is odd for $E\, \in\, {\mathcal M}_X(r,d)$.

There is a quaternionic universal vector bundle over
$X\times {\mathcal M}_X(r,d)$ if and only if the degree $d$
is odd.
\end{theorem}

See Propositions \ref{prop-existence-2}, 
\ref{prop3}, \ref{prop4}, and 
\ref{propquat} for a proof of Theorem \ref{thm2}.

A similar result holds for moduli spaces of vector bundles
with fixed determinant; see Section \ref{sec4}.

We note that, somewhat ironically, one can have a universal real (or 
quaternionic) bundle on $X\times {\rm Pic}^d(X)$ or $X\times {\mathcal 
M}_X(r,d)$ without there being a corresponding real (or quaternionic) 
bundle over $X$. Indeed, the constraint for the existence of a real 
bundle on $X$ is that the degree $d$ be even; for a quaternionic bundle, 
$\chi(E)$ must be even (see \cite{BHH}); these can be contrasted with 
the conditions in the theorems above.

Instead of the involution $E\,\longrightarrow\,\sigma^*\overline{E}$,
one may also consider the involution of the moduli space
defined by $E\,\longrightarrow\,\sigma^*\overline{E}^*$. Note that
$\text{degree}(E)\,=\, 0$ if $E$ and $\sigma^*\overline{E}^*$ lie
in the same moduli space. Unitary flat vector bundles on nonorientable
surfaces are extensively studied; see \cite{HJ}, \cite{HL}. The
holomorphic vector bundles corresponding to these flat bundles
are fixed points of the involution $E\,\longrightarrow\,
\sigma^*\overline{E}^*$. Since there is no universal vector bundle
over $X\times {\mathcal M}_X(r,0)$ if $r\, \geq\, 2$, there is
no real or quaternionic universal vector bundle over
$X\times {\mathcal M}_X(r,0)$ for this involution.

\medskip
\noindent
\textbf{Acknowledgements.}\, The first--named author would like
to thank McGill University for hospitality.

\section{Real curves and universal bundles}

\subsection{Universal bundles}\label{s2.1}
Let $Y$ be a compact connected Riemann surface.
Let $J(Y)$ be the Jacobian of $Y$, which is an
abelian variety that parametrizes the
isomorphism classes of topologically trivial holomorphic
line bundles over $Y$. 
For any integer $d$, let $\text{Pic}^d(Y)$ denote the torsor
for $J(Y)$ that parametrizes the isomorphism classes of
holomorphic line bundles over $Y$ of degree $d$. A \textit{Poincar\'e
line bundle} over $Y\times \text{Pic}^d(Y)$ is a holomorphic
line bundle
$$
{\mathcal L}\, \longrightarrow\, Y\times \text{Pic}^d(Y)
$$
such that for each point $z\, \in\, \text{Pic}^d(Y)$ the line
bundle ${\mathcal L}\vert_{Y\times\{z\}}\,\longrightarrow\, Y$
is in the isomorphism class parametrized by $z$. There is
a Poincar\'e line bundle over $Y\times \text{Pic}^d(Y)$. If we
fix a point $y_0\, \in\, Y$, then there is a unique, up to an
isomorphism, Poincar\'e line bundle over $Y\times \text{Pic}^d(Y)$
which is trivial over $\{y_0\}\times \text{Pic}^d(Y)$;
see \cite[Ch. IV, \S~2]{ACGH} for the details. A Poincar\'e
line bundle is also called a universal line bundle.

Let $g$ be the genus of $Y$.

Fix an integer $r\, \geq\, 2$; let $d$ be any integer.
Let ${\mathcal M}_Y(r,d)$ denote the moduli space of stable
vector bundles over $Y$ of rank $r$ and degree $d$. This moduli
space, which was first constructed by Mumford, \cite{Mu}, is an
irreducible smooth quasi--projective variety over $\mathbb C$.
If $g\,=\, 0$, then ${\mathcal M}_Y(r,d)$ is an empty set;
this follows immediately from a theorem of Grothendieck which
says that any holomorphic vector bundle over ${\mathbb C}
{\mathbb P}^1$ splits into a direct sum of
holomorphic line bundles \cite{Gr}.
If $g\,=\, 1$, and $\text{g.c.d.}(r\, ,d)\, \not=\, 1$,
then ${\mathcal M}_Y(r,d)$ is an empty set; if
$\text{g.c.d.}(r\, ,d)\, =\, 1$, then, up to tensoring with
a line bundle, there is exactly one stable vector bundle
over $Y$ of rank $r$ and degree $d$ (see \cite{At1}).
If $g\, \geq\, 2$, then the dimension of
${\mathcal M}_Y(r,d)$ is $r^2(g-1)+1$.

When we consider the moduli space ${\mathcal M}_Y(r,d)$,
it will always be assumed that $g\, \geq\, 2$.

An algebraic vector bundle
$$
{\mathcal E}\, \longrightarrow\, Y\times {\mathcal M}_Y(r,d)
$$
is called \textit{universal} if for each point $z\, \in\,
{\mathcal M}_Y(r,d)$, the vector bundle ${\mathcal E}\vert_{Y\times
\{z\}}\,\longrightarrow\, Y$ is in the isomorphism class parametrized
by $z$.

There is a universal vector bundle over $Y\times {\mathcal 
M}_Y(r,d)$ if and only if $\text{g.c.d.}(r\, ,d)\, =\, 1$
\cite{Ra} (see also
\cite{Ne}). In Section \ref{s2.3} we recall a construction
of the universal vector bundle under the assumption that the degree
is coprime to the rank.

\subsection{Real curves; real and quaternionic bundles}

Let $X_{\mathbb R}$ be an irreducible smooth projective curve
defined over the field of
real numbers. The corresponding complex curve
$X\, :=\, X_{\mathbb R}\times_{\mathbb R}{\mathbb C}$ is a
compact Riemann surface. The automorphism of ${\mathbb C}$
defined by $z\, \longmapsto\, \overline{z}$ produces an
anti--holomorphic involution
\begin{equation}\label{p0}
\sigma\, :\, X\, \longrightarrow\, X\, .
\end{equation}

If $X_{\mathbb R}$ is geometrically irreducible, then $X$ is
connected. If $X_{\mathbb R}$ is not geometrically
irreducible, then $X$ is a disjoint union
\begin{equation}\label{p1}
X \,=\, S\cup \overline{S}\, ,
\end{equation}
where $S$ is a compact connected Riemann surface, and
$\overline{S}$ is the conjugate of $S$, meaning the smooth
manifold underlying $\overline{S}$ coincides with that of
$S$, while the almost complex structure of $\overline{S}$ is
$-J_S$, where $J_S$ is the almost complex structure of $S$.
So the identity map $S\,\longrightarrow\, \overline{S}$ is
anti--holomorphic, and this map coincides with $\sigma$.

Let $Z$ be a variety defined over $\mathbb R$; let
$Z_{\mathbb C}\, :=\, Z\times_{\mathbb R}\mathbb C$ be the
corresponding complex variety. Let
$$
\sigma_Z\, :\, Z_{\mathbb C}\, \longrightarrow\,
Z_{\mathbb C}
$$
be the anti--holomorphic involution given by the automorphism
of ${\mathbb C}$ defined by $z\, \longmapsto\, \overline{z}$.
Giving a real algebraic
vector bundle over $Z$ is equivalent to giving a complex
algebraic vector bundle $E\,\longrightarrow\,Z_{\mathbb C}$
together with an algebraic isomorphism of vector bundles
\begin{equation}\label{eta}
\eta\, :\, E\, \longrightarrow\,
\sigma^*_Z\overline{E}
\end{equation}
such that the composition
$$
E\, \stackrel{\eta}{\longrightarrow}\,\sigma^*_Z\overline{E}\,
\stackrel{\sigma^*_Z\overline{\eta}}{\longrightarrow}\,
\sigma^*_Z\overline{\sigma^*_Z\overline{E}}\, =\, E
$$
is the identity map of $E$. Note that this condition
means that there is an involution lifting $\sigma_Z$.

The above correspondence between real
algebraic vector bundles on $Z$ and complex algebraic vector bundles
on the complexification $Z_{\mathbb C}$
equipped with an involution lifting
$\sigma_Z$ will be used throughout without further clarification.

For convenience, sometimes a real
algebraic vector bundle over $Z$ will also be called
as a real algebraic vector bundles on $Z_{\mathbb C}$ (this happens
when we refer to both real and quaternionic vector bundles in a
sentence).

In the same vein, a \textit{quaternionic} vector bundle
on $Z_{\mathbb C}$ is a complex algebraic vector bundle
$$
E\,\longrightarrow\, Z_{\mathbb C}
$$
together with an algebraic isomorphism
$$
\eta\, :\, E\, \longrightarrow\, \sigma^*_Z\overline{E}
$$
such that the composition
$$
E\, \stackrel{\eta}{\longrightarrow}\,\sigma^*_Z\overline{E}\,
\stackrel{\sigma^*_Z\overline{\eta}}{\longrightarrow}\, 
\sigma^*_Z\overline{\sigma^*_Z\overline{E}}\, =\, E
$$
is multiplication by $-1$.

Note that the quaternionic vector bundles are defined only on the
complexification.

For any integer $d$, let $\text{Pic}^d(X_{\mathbb R})_{\mathbb C}
\,:=\, \text{Pic}^d(X_{\mathbb R})\times_{\mathbb R}
\mathbb C$ be the complexification of the
Picard variety of $X_{\mathbb R}$. So, if $X_{\mathbb R}$ is
geometrically irreducible, then
$$
\text{Pic}^d(X_{\mathbb R})_{\mathbb C}
\,=\, \text{Pic}^d(X)\, .
$$
If $X_{\mathbb R}$ is not
geometrically irreducible, and $d$ is odd, then
$\text{Pic}^d(X_{\mathbb R})_{\mathbb C}$ is the empty
set. If $X_{\mathbb R}$ is not
geometrically irreducible, and $d$ is even, then
$$
\text{Pic}^d(X_{\mathbb R})_{\mathbb C}
\,=\, \text{Pic}^{d/2}(S)\times \text{Pic}^{d/2}
(\overline{S})\, ,
$$
where $S$ and $\overline{S}$ are as in \eqref{p1}.

Let
\begin{equation}\label{p2}
\widetilde{\sigma}\, :\, \text{Pic}^d(X_{\mathbb R})_{\mathbb C}
\,\longrightarrow\,\text{Pic}^d(X_{\mathbb R})_{\mathbb C}
\end{equation}
be the anti--holomorphic diffeomorphism that sends any
holomorphic line bundle $L$ on $X$ to 
$\sigma^* \overline{L}$,
where $\sigma$ is the anti--holomorphic involution in \eqref{p0}.
This $\widetilde{\sigma}$ is clearly an involution, and it
coincides with the involution given by the automorphism
$z\,\longmapsto\,\overline{z}$ of $\mathbb C$.

Note that both quaternionic and real line bundles over $X$ of degree
$d$ are represented by real points of $\text{Pic}^d(X_{\mathbb 
R})_{\mathbb C}$.

Now we consider stable vector bundles of higher ranks.

If $X_{\mathbb R}$ is geometrically irreducible,
an algebraic vector bundle $E$ over $X_{\mathbb R}$ will
be called \textit{stable} if for every proper subbundle $F\,
\,\subset\, E$ of positive rank, the following inequality
holds:
$$
\text{degree}(F)/\text{rank}(F)\, <\,
\text{degree}(E)/\text{rank}(E)\, .
$$
So if $E$ is stable, then the corresponding vector bundle
$E\bigotimes_{\mathbb R}\mathbb C$ over $X$ is polystable. For
a real algebraic vector bundle $E\,\longrightarrow\, X_{\mathbb
R}$, from the uniqueness of the Harder--Narasimhan filtration
it follows that each term in the Harder--Narasimhan filtration
of $E\bigotimes_{\mathbb R} \mathbb C$ over $X_{\mathbb R}
\times_{\mathbb R} \mathbb C$ is real, meaning it is preserved
by the involution of $E\bigotimes_{\mathbb R} \mathbb C$.
We note that if $d$ is coprime to $r$, then $E$ is stable if and only
if $E\bigotimes_{\mathbb R}\mathbb C$ on $X$ is stable.

If $X_{\mathbb R}$ is not geometrically irreducible,
an algebraic vector bundle $E$ over $X_{\mathbb R}$ will
be called \textit{stable} if the restriction of
$E\bigotimes_{\mathbb R}\mathbb C$ to the component $S$ 
in \eqref{p1} is stable. Note that the restriction of 
$E\bigotimes_{\mathbb R}\mathbb C$ to $S$ is stable if and
only if the restriction of $E\bigotimes_{\mathbb R}
\mathbb C$ to $\overline{S}$ is stable.

Let $g\, :=\, H^1(X_{\mathbb R},\,{\mathcal O}_{X_{\mathbb R}})$
be the genus of $X_{\mathbb R}$. As mentioned before,
while considering higher rank bundles it is assumed that
$g\, \geq\,2$. We further assume that $g\, \geq\,4$ if
$X_{\mathbb R}$ is not geometrically irreducible. This
condition is equivalent to the condition that the genus
of $S$ in \eqref{p1} is at least two.

Let ${\mathcal M}_{X_{\mathbb R}}(r,d)$ be the moduli
space of stable vector bundles over $X_{\mathbb R}$
of rank $r$ and degree $d$. Let
\begin{equation}\label{wm}
\widetilde{\mathcal M}_X(r,d)\,=\,{\mathcal M}_{X_{\mathbb R}}
(r,d) \times_{\mathbb R}\mathbb C
\end{equation}
be the complexification. Assume that
$X_{\mathbb R}$ is geometrically irreducible.
Let ${\mathcal M}_X(r,d)$
be the moduli space of stable vector bundles over $X$
of rank $r$ and degree $d$. From the
construction of the moduli space it follows that
${\mathcal M}_X(r,d)$ is a Zariski open subset
of $\widetilde{\mathcal M}_X(r,d)$. But if $d$ is coprime
to $r$, then ${\mathcal M}_X(r,d)$ coincides with
$\widetilde{\mathcal M}_X(r,d)$.

If $X_{\mathbb R}$ is not
geometrically irreducible, and $d$ is odd, then
$\widetilde{\mathcal M}_X(r,d)$ is the empty set. 
If $X_{\mathbb R}$ is not
geometrically irreducible, and $d$ is even, then
$$
\widetilde{\mathcal M}_X(r,d)\,=\, {\mathcal M}_S(r,d/2)\times
{\mathcal M}_{\overline{S}}(r,d/2)\, ,
$$
where $S$ and $\overline{S}$ are as in \eqref{p1}; recall
that $\text{genus}(S)\,\geq\,2$. Let
\begin{equation}\label{p3}
\widetilde{\sigma}\, :\,\widetilde{\mathcal M}_X(r,d)
\,\longrightarrow\,\widetilde{\mathcal M}_X(r,d)
\end{equation}
be the anti--holomorphic diffeomorphism defined by
$E\, \longmapsto\, \sigma^* \overline{E}$, where
$\sigma$ is the involution in \eqref{p0}.
We note that $\widetilde{\sigma}$ coincides with the
involution of $\widetilde{\mathcal M}_X(r,d)$
given by the automorphism
$z\,\longmapsto\,\overline{z}$ of $\mathbb C$.

Note that both quaternionic and real stable vector
bundles over $X$ of rank $r$ and degree $d$
are represented by real points of
$\widetilde{\mathcal M}_X(r,d)$.

\subsection{Real and quaternionic universal bundles}\label{s2.3}

First assume that $X_{\mathbb R}$ is geometrically
irreducible.
A \textit{real universal line bundle} over $X\times 
\text{Pic}^d(X)$ is a universal line bundle
$$
{\mathcal L}\, \longrightarrow\, X\times \text{Pic}^d(X)
$$
equipped with a holomorphic
$$
\sigma_P\,:\, 
{\mathcal L}\,\longrightarrow\, 
(\sigma\times\widetilde{\sigma})^*\overline{\mathcal L}\, ,
$$
 where $\sigma$ and $\widetilde\sigma$ are the involutions
in \eqref{p0} and \eqref{p2} respectively, such that
$\sigma_P\circ \sigma_P\,=\, \text{Id}_{\mathcal L}$.
The definition of a quaternionic universal bundle is the same, except 
that $\sigma_P\circ \sigma_P\,=\, -\text{Id}_{\mathcal L}$.

Now assume that $X_{\mathbb R}$ is not geometrically
irreducible. Take $d$ to be even, say $d\,=\,2d_0$ (recall
that $\text{Pic}^d(X_{\mathbb R})_{\mathbb C}$ is the empty
set when $d$ is odd).

A \textit{real universal line bundle} over $X\times 
\text{Pic}^d(X_{\mathbb R})_{\mathbb C}$ is a universal
line bundle
$$
{\mathcal L}\, \longrightarrow\, 
S\times \text{Pic}^{d_0}(S)\, ,
$$
where $S$ is the Riemann surface in \eqref{p1}. Note that
the universal line bundle over $\overline{S}\times
\text{Pic}^{d_0}(\overline{S})$ is uniquely determined by
${\mathcal L}$; more precisely, it is the pullback $(\sigma
\times\widetilde{\sigma})^*\overline{\mathcal L}$.

Since a universal line bundle over $S\times \text{Pic}^{d_0}(S)$
exists, we conclude that there is a real
(or quaternionic; we just have to change sign of the
isomorphism over the component $S\times \text{Pic}^{d_0}(S)$
of $(S\times \text{Pic}^{d_0}(S))\bigcup (\overline{S}
\times \text{Pic}^{d_0}(\overline{S}))$) universal line bundle
over $X\times \text{Pic}^d(X_{\mathbb R})_{\mathbb C}$
if $X_{\mathbb R}$ is not geometrically irreducible.

The real points of $X_{\mathbb R}$ are the fixed points of
the involution $\sigma$ of $X$. Assume that $X_{\mathbb R}$ has
a real point $x_0$ (note that this implies that $X_{\mathbb R}$
is geometrically irreducible). As mentioned
in Section \ref{s2.1}, there is a unique, up to
an isomorphism, universal line bundle over
$X\times \text{Pic}^d(X)$ which is trivial on
$\{x_0\}\times\text{Pic}^d(X)$. Let ${\mathcal L}_d\,
\longrightarrow\, X\times \text{Pic}^d(X)$ be this universal
line bundle. Fix a trivialization of ${\mathcal L}_d$ over
$\{x_0\}\times\text{Pic}^d(X)$. From the uniqueness of
${\mathcal L}_d$ it follows that there is a unique involution
$$
{\mathcal L}_d\, \longrightarrow\, (\sigma\times
\widetilde{\sigma})^*\overline{{\mathcal L}_d}
$$
whose restriction to $\{x_0\}\times\text{Pic}^d(X)$ is
the conjugation of $\mathbb C$ with respect to the
chosen trivialization;
the involution $\widetilde{\sigma}$ is constructed in \eqref{p2}.
In other words, ${\mathcal L}_d$ gets a real structure.
Consequently, if $\sigma$
has a fixed point, there is a real universal line bundle
over $X\times \text{Pic}^d(X)$.

Now we consider moduli spaces of stable vector bundles of
higher rank.

First assume that $X_{\mathbb R}$ is not geometrically
irreducible. Take the integer $d$ to be even, say $d\,=\,2d_0$ (recall
that $\widetilde{\mathcal M}_X(r,d)$ in \eqref{wm} is the empty
set when $d$ is odd).

A \textit{real universal vector bundle} over $X\times
\widetilde{\mathcal M}_{X}(r,d)$ is a universal vector bundle
$$
{\mathcal E}\, \longrightarrow\, S\times{\mathcal M}_S(r,d_0)\, ,
$$
where $S$ in the Riemann surface in \eqref{p1}; as before,
$\mathcal E$ gives the universal vector bundle
$(\sigma\times \widetilde{\sigma})^*\overline{\mathcal E}$
over $\overline{S}\times{\mathcal M}_{\overline{S}}(r,d_0)$. 

As mentioned in Section \ref{s2.1}, there is a universal
vector bundle over $S\times{\mathcal M}_S(r,d_0)$
if and only if $d_0$ is coprime to $r$.
Hence, if $X_{\mathbb R}$ is not geometrically irreducible,
there is a real (or quaternionic) universal vector bundle over $X\times 
\widetilde{\mathcal M}_{X}(r,2d_0)$ if and only if
$d_0$ is coprime to $r$.

Now assume that $X_{\mathbb R}$ is geometrically
irreducible.
A \textit{real universal vector bundle} over $X\times
{\mathcal M}_{X}(r,d)$ is a universal vector bundle
$$
{\mathcal E}\, \longrightarrow\, X\times{\mathcal M}_X(r,d)
$$
equipped with an conjugate--linear involution
$$
{\mathcal E}\,\longrightarrow\, {\mathcal E}\, ,
$$
lifting $\sigma\times\widetilde{\sigma}$, where 
$\sigma$ and $\widetilde\sigma$ are the involutions
constructed in \eqref{p0} and \eqref{p3} respectively. The
square of this lift must be the identity; for a universal
quaternionic vector bundle, one has the same type
of lift, but now one asks that the square be minus 
the identity.

We will show that if $\sigma$ has a fixed point, and $d$
is coprime to $r$, then there is a 
real universal vector bundle over $X\times 
{\mathcal M}_{X}(r,d)$. For that we need to recall a
construction of the universal vector bundle over the complex
numbers under the assumption that $d$ is coprime to $r$.

As before, assume that $X_{\mathbb R}$ is geometrically
irreducible. Assume that $d$ is coprime to $r$. Set
\begin{equation}\label{chi}
\chi\, :=\, d+r(1-g)\, ,
\end{equation}
where $g$ is the genus of $X$.
Note that $\chi$ is the Euler characteristic of any vector
bundle on $X$ lying in ${\mathcal M}_X(r,d)$. Since $d$
is coprime to $r$, there are integers $a$ and $b$ such that
\begin{equation}\label{u1}
ar+b\chi\, =\, -1\, .
\end{equation}

The moduli space ${\mathcal M}_X(r,d)$
is a quotient of a Quot scheme $\mathcal Q$ (\cite{Mu}). Let
$$
{\mathcal E}\, \longrightarrow\, X\times \mathcal Q
$$
be the tautological universal vector bundle.
Fix a point $x_0$ in $X$. Let
$$
{\mathcal E}_{x_0}\, \longrightarrow\, \{x_0\}\times
{\mathcal Q}\, =\, \mathcal Q
$$
be the restriction. Let
$$
{\mathcal D}\, :=\, \det R^0p_{*} {\mathcal E}
\otimes(\det R^1p_{*} {\mathcal E})^*\,
\longrightarrow\,\mathcal Q
$$
be the determinant line bundle, where $p\, :\,
X\times \mathcal Q \, \longrightarrow\, \mathcal Q$
is the natural projection. Then the vector bundle
\begin{equation}\label{u2}
{\mathcal E}\otimes
(\bigwedge\nolimits^r p^*{\mathcal E}_{x_0})^{\otimes a}
\otimes p^*{\mathcal D}^{\otimes b}\, \longrightarrow\,
X\times \mathcal Q
\end{equation}
descends to a universal vector bundle over
$X\times {\mathcal M}_X(r,d)$, where $a$ and $b$ are
as in \eqref{u2}. From \eqref{u1} it follows that
the multiplication action of the nonzero scalars on
the vector bundle in \eqref{u2} is trivial.
We recall that the vector bundle
in \eqref{u2} descends to $X\times {\mathcal M}_X(r,d)$
because ${\mathbb C}^*$ acts trivially on it.

The Quot scheme $\mathcal Q$ can be so chosen that it has a
real structure; indeed one is looking at quotients of a fixed real 
bundle, which can be chosen to be real. Then the vector bundle
${\mathcal E}\, \longrightarrow\, X\times \mathcal Q$
has a real structure. If the point $x_0$ is fixed by
$\sigma$, the universal vector bundle over
$X\times {\mathcal M}_X(r,d)$ constructed above has
a real structure.

If $a$ in \eqref{u2} is an even integer, then the
vector bundle in \eqref{u2} can be replaced by
$$
{\mathcal E}\otimes
(\bigwedge\nolimits^r p^*{\mathcal E}_{x_0})^{\otimes a/2}
\otimes
(\bigwedge\nolimits^r p^*{\mathcal E}_{\sigma(x_0}))^{\otimes a/2}
\otimes {\mathcal D}^{\otimes b}\, \longrightarrow\,
X\times \mathcal Q
$$
which also descends to $X\times {\mathcal M}_X(r,d)$ as a
universal vector bundle for the same reason. If
the Quot scheme has a real structure, this descended vector
bundle also has a real structure.

So we have the following lemma.

\begin{lemma}\label{lem4}
Assume that $X_{\mathbb R}$ is geometrically
irreducible. Also, assume that $d$ is coprime to $r$.

If $\sigma$ has a fixed point, then there is a real universal
vector bundle over $X\times {\mathcal M}_X(r,d)$.

If the integer $a$ in \eqref{u1} is even, then
there is a real universal
vector bundle over $X\times {\mathcal M}_X(r,d)$.
\end{lemma}

One could try the same approach for quaternionic structures: the 
universal bundle over the Quot scheme can be made quaternionic, but 
unfortunately the technique for having it descend to the moduli
space does not go through.

Henceforth, we will always make the following two assumptions:
\begin{enumerate}
\item $X_{\mathbb R}$ is geometrically
irreducible, and

\item $X_{\mathbb R}$ does not have any real points.
\end{enumerate}

We have seen above that the question of the existence of
a real universal bundle is settled if any of the above
two conditions fails.

\section{Line bundles}\label{sec2}

\subsection{Existence of universal bundles}

The following lemma gives an explicit universal line bundle.

\begin{lemma}\label{lem1}
There is a natural
real universal line bundle over $X\times {\rm Pic}^{g-2}(X)$.
\end{lemma}

\begin{proof}
On $\text{Pic}^{g-1}(X)$, there is a canonical theta hypersurface
$\Theta$ that parametrizes all line bundles
$\zeta\,\longrightarrow\, X$
of degree $g-1$ such that $H^0(X,\, \zeta)\,\not=\, 0$. If
$$
{\mathcal L}_{g-1}\, \longrightarrow\,
X\times \text{Pic}^{g-1}(X)
$$ is a universal line bundle, then
the line bundle ${\mathcal O}_{\text{Pic}^{g-1}(X)}(-\Theta)$ is
the determinant line bundle
$$
\det {\mathcal L}_{g-1}\, :=\,
(\det p_{2*}{\mathcal L}_{g-1})\otimes
(\det R^1 p_{2*}{\mathcal L}_{g-1})^*\,\longrightarrow\,
\text{Pic}^{g-1}(X)\, ,
$$
where $p_2$ is the projection of $X\times \text{Pic}^{g-1}(X)$ to
$\text{Pic}^{g-1}(X)$. Since $\chi(L)\,=\,0$ for any
$L\,\in\, \text{Pic}^{g-1}(X)$, the determinant line bundle
$\det {\mathcal L}_{g-1}$ is independent of the choice
of ${\mathcal L}_{g-1}$.

Let
\begin{equation}\label{phi2}
\phi\, :\, X\times {\rm Pic}^{g-2}(X)\, \longrightarrow\,
{\rm Pic}^{g-1}(X)
\end{equation}
be the map defined by $(x\, ,L)\, \longmapsto\, L
\bigotimes{\mathcal O}_X(x)$. Define a line bundle
\begin{equation}\label{b1}
{\mathcal L}\, :=\,
\phi^*{\mathcal O}_{{\rm Pic}^{g-1}(X)}(-\Theta)\otimes p^*_X K_X
\,\longrightarrow\, X\times {\rm Pic}^{g-2}(X),
\end{equation}
where $K_X$ is the holomorphic cotangent bundle of $X$,
and $p_X$ is the projection of $X\times {\rm Pic}^{g-2}(X)$
to $X$. We will show that $\mathcal L$ is a universal line bundle.

For that, take any holomorphic line bundle
$$
\zeta\,\longrightarrow\, X
$$
of degree $g-2$. Consider the line bundle
$$
\widetilde{\zeta}\,:=\, 
(q^*_1 \zeta)\otimes {\mathcal O}_{X\times X}(\Delta)
\,\longrightarrow\, X\times X\, ,
$$
where $q_i$, $i\,=\, 1\, ,2$, is the projection of
$X\times X$ to the $i$--th factor, and $\Delta\,\subset\,
X\times X$ is the diagonal divisor. From the above
mentioned identification of ${\mathcal
O}_{\text{Pic}^{g-1}(X)}(-\Theta)$ as a determinant line bundle
it follows that
\begin{equation}\label{b0}
\phi^*{\mathcal O}_{{\rm 
Pic}^{g-1}(X)}(-\Theta)\vert_{X\times \{\zeta\}}
\,=\,
\det \widetilde{\zeta}\,:=\, (\det R^0q_{2*}\widetilde{\zeta})
\otimes (\det R^1q_{2*}\widetilde{\zeta})^*\, ,
\end{equation}
where $\phi$ is the map in \eqref{phi2}.

In view of \eqref{b0}, to prove that
$\mathcal L$ is a universal line bundle it suffices
to show that
\begin{equation}\label{z1}
\zeta\,=\, (\det \widetilde{\zeta})\otimes K_X\, .
\end{equation}
To prove \eqref{z1}, we use \cite[p. 368, Lemma 6]{BM}
for the family of line bundles $\widetilde{\zeta}$
over $X$ parametrized by the second factor of $X\times X$.
Given a pair of algebraic line bundles $\eta_1$ and $\eta_2$ on an
algebraic family of curves, the Deligne pairing $\langle \eta_1\, ,
\eta_2\rangle$ is a line bundle on the parameter space
(see \cite{De}, \cite{BM} for its construction). The line bundle
$\langle q^*_1 \zeta\, ,{\mathcal O}_{X\times X}(\Delta) \rangle$
is isomorphic to $\zeta$ (see
\cite[p. 367, Proposition 5(c)]{BM}). Now \cite[p. 368, Lemma 6]{BM}
says that
\begin{equation}\label{b2}
\zeta\,=\, (\det \widetilde{\zeta})\otimes
(\det {\mathcal O}_{X\times X}(\Delta))^*
\end{equation}
(all other line bundles in \cite[p. 368, Lemma 6]{BM} are trivial
in our case).

Consider the short exact sequence of sheaves on $X\times X$
\begin{equation}\label{b3}
0\,\longrightarrow\, {\mathcal O}_{X\times X}
\,\longrightarrow\, {\mathcal O}_{X\times X}(\Delta)
\,\longrightarrow\, {\mathcal O}_{X\times X}(\Delta)\vert_\Delta
\,\longrightarrow\, 0\, .
\end{equation}
The Poincar\'e adjunction formula says that the restriction of
the line bundle ${\mathcal O}_{X\times X}(\Delta)$ to $\Delta$
is the tangent bundle $TX$. Hence from \eqref{b3},
$$
\det {\mathcal O}_{X\times X}(\Delta)\,=\, TX\, ,
$$
where $TX$ is the holomorphic tangent bundle. Now \eqref{b2}
implies \eqref{z1}. Therefore, $\mathcal L$
is a universal line bundle over $X\times{\rm Pic}^{g-2}(X)$.

{}From the construction of $\mathcal L$
it follows immediately that there is a natural involution
of $\mathcal L$ lifting the anti--holomorphic involution
$\sigma\times\widetilde{\sigma}$ of $X\times
{\rm Pic}^{g-2}(X)$. Therefore, $\mathcal L$
is a real universal line bundle. This completes the
proof of the lemma.
\end{proof}

We recall that a real (respectively, quaternionic) algebraic line bundle 
on $X_{\mathbb R}$
is a holomorphic line bundle on $X$ equipped with an involution
lifting $\sigma$, whose square is the identity (respectively,
minus the identity). 

\begin{lemma}\label{lem2}
Let $d$ be an integer such that there
is a real universal line bundle on $X\times {\rm Pic}^{d}(X)$.
Let $L$ be a real (respectively, quaternionic) line bundle on $X$
of degree $d_1$. Then there is a real (respectively, quaternionic)
universal line bundle on $X\times {\rm Pic}^{d+d_1}(X)$. 

Similarly, let $d$ be an integer such that there
is a quaternionic universal line bundle on $X\times {\rm Pic}^{d}(X)$.
Let $L$ be a real (respectively, quaternionic) line bundle on $X$
of degree $d_1$. Then there is a quaternionic (respectively, real)
universal line bundle on $X\times {\rm Pic}^{d+d_1}(X)$. 
\end{lemma}

\begin{proof}
Let ${\mathcal L}\, \longrightarrow\, X\times {\rm Pic}^{d}(X)$
be a real universal line bundle. Let
$$
\phi\, :\, {\rm Pic}^{d+d_1}(X)\, \longrightarrow\,
{\rm Pic}^{d}(X)
$$
be the morphism defined by $L'\, \longmapsto\, L'\bigotimes L^*$.
Then $(\text{Id}_X\times\phi)^*{\mathcal L}\bigotimes p^*_X L$,
where $p_X$ is the projection of $X\times {\rm Pic}^{d+d_1}(X)$ to $X$, 
is naturally a real (respectively, quaternionic) universal line
bundle if $L$ is real (respectively, quaternionic).
The same procedure works for the second half of the proposition.
\end{proof}

For any point $x_0\, \in\, X$, the line bundle
${\mathcal O}_X(x_0+\sigma(x_0))$ is real; taking duals, varying $x_0$ 
and tensoring, there are then real line bundles on $X$ in any even
degree. There are none in odd degree. In the same vein, one can show that 
quaternionic line bundles exist in degree $d$ if and only if $d-g+1 
\equiv 0$ mod $2$ [BHH], \cite[p. 55, Theorem 2.6]{AB}. For example, there 
is a quaternionic theta characteristic on $X$ 
if $g$ is even (see \cite[pp. 61--62]{At2}; recalling that $X_{\mathbb 
R}$ does not have any real point, there is no real theta 
characteristic on $X$ because there is no real line bundle on $X$ of odd 
degree).

\begin{proposition}\label{propexistence}
If the integer $d-g$ is even, then there
is a real universal line bundle over $X\times {\rm Pic}^d(X)$.

If the degree $d$ is odd, then there
is a quaternionic universal line bundle over $X\times {\rm Pic}^d(X)$.
\end{proposition}

\begin{proof}
Indeed, Lemma \ref{lem1} and Lemma \ref{lem2} allow us to produce all the 
universal bundles in the theorem from the real one in degree $g-2$, by 
twisting by an appropriate real or quaternionic line bundle. We note
that there is a real line bundle on $X$ of degree two. Also,
there is a quaternionic line bundle on $X$ of degree one (respectively,
zero) if $g$ is even (respectively, odd)
\cite[Theorem 6.6]{BHH} (see also \cite[p. 55, Theorem 2.6]{AB}).
\end{proof}

\subsection{Non-existence: the case of even genus}

First assume that the genus $g$ is an even integer.

\begin{proposition}\label{prop1}
There is no real universal line bundle on $X\times {\rm Pic}^{
2d+1}(X)$.

There is no quaternionic universal line bundle on $X\times {\rm Pic}^{
2d}(X)$.
\end{proposition}

\begin{proof}
For any point $x\,\in\, X$, the line bundle
${\mathcal O}_X(x+\sigma(x))$ is real. Therefore,
in view of Lemma \ref{lem2}, it suffices to prove
the first part of the proposition for $d\,=\,1$.
 
Fix a quaternionic line bundle
\begin{equation}\label{f1}
L_0\, \longrightarrow\, X
\end{equation}
such that $\text{degree}(L_0)\, =\,1$.
Let
$$
{\mathcal L}\, \longrightarrow\, X\times {\rm Pic}^{0}(X)
$$
be a real universal line bundle; from Proposition \ref{propexistence}
we know that such a universal line bundle exists. Let
$$
p_X\, :\, X\times {\rm Pic}^{1}(X)\, \longrightarrow\, X
$$
be the natural projection. We note that
\begin{equation}\label{f3}
{\mathcal L}_1\, :=\,
p^*_X L_0\otimes (\text{Id}_X\times \varphi)^*{\mathcal L}
\longrightarrow\,X\times {\rm Pic}^1(X)
\end{equation}
is a quaternionic universal line bundle.

To see that there is no real universal line bundle
on $X\times {\rm Pic}^1(X)$, let
$$
p_2\, :\, X\times {\rm Pic}^{1}(X)\, \longrightarrow\,
{\rm Pic}^{1}(X)
$$
be the natural projection.
Any universal line bundle on $X\times {\rm Pic}^1(X)$ is of the
form
\begin{equation}\label{f4}
{\mathcal L}_1\otimes p^*_2 \xi\, ,
\end{equation}
where ${\mathcal L}_1$ is the line bundle in \eqref{f3}.
If the line bundle ${\mathcal L}_1\bigotimes p^*_2 \xi$ in \eqref{f4}
is real, then $\xi$ must be a quaternionic line bundle, because
${\mathcal L}_1$ is quaternionic. On the
other hand, it can be shown that there are no quaternionic line
bundles on ${\rm Pic}^1(X)$. Indeed, the variety ${\rm Pic}^1
(X_{\mathbb R})$ defined over $\mathbb R$
has a real point because the line bundle $L_0$
in \eqref{f1} is represented by a real point of ${\rm Pic}^1
(X_{\mathbb R})$; hence there is no quaternionic line bundle
on ${\rm Pic}^1(X)$ (see \cite[p. 201, Proposition 2.7.4]{CP}).

For the second part, if one had a universal quaternionic bundle in even 
degree, one could use Lemma \ref{lem2} to produce a universal real bundle 
in odd degree because there is a quaternionic line bundle on $X$ of
degree one \cite[Theorem 6.6]{BHH}; we
have just proven that these don't exist. This completes 
the proof of the proposition.
\end{proof}

\subsection{Non-existence; the case of odd genus}

Now assume that $g$ is odd.

\begin{proposition}\label{prop2}
There is no real or quaternionic universal line bundle on
$X\times {\rm Pic}^{2d}(X)$.
\end{proposition}

\begin{proof}
In view of Lemma \ref{lem2}, it suffices to prove
the proposition for $d\,=\, 0$.

Let
\begin{equation}\label{e0}
{\mathcal L}'\, \longrightarrow\, X\times {\rm Pic}^{0}(X)
\end{equation}
be a real or quaternionic universal line bundle. We will
show that ${\mathcal L}'$ can be modified to construct another
real or quaternionic universal line bundle which is topologically
trivial in the direction of ${\rm Pic}^{0}(X)$.

Let
$$
c\, \in\, H^2({\rm Pic}^{0}(X),\, {\mathbb Z})
$$
be the K\"unneth component of the Chern class $c_1({\mathcal L}')
\,\in\, H^2(X\times {\rm Pic}^{0}(X),\, {\mathbb Z})$
of the line bundle in \eqref{e0}. Since ${\mathcal L}'$ is real or
quaternionic, the anti--holomorphic involution of ${\rm Pic}^{0}(X)$
(see \eqref{p2}) preserves $c$. Hence there is a real line bundle
\begin{equation}\label{e1}
\xi\, \longrightarrow\, {\rm Pic}^{0}(X)
\end{equation}
such that $c_1(\xi)\, =\, c$ (see the paragraph
following Remark 3.3.2 in \cite[p. 215]{CP}); note that the
trivial line bundle ${\mathcal O}_X$ is fixed by the anti--holomorphic
involution of ${\rm Pic}^{0}(X)$, so the assumption in \cite[p. 214,
\S~3.3]{CP} is fulfilled.

Let $p_2\, :\, X\times {\rm Pic}^{0}(X)\, \longrightarrow\, {\rm 
Pic}^{0}(X)$ be the natural projection. Consider the line bundle
\begin{equation}\label{e2}
{\mathcal L}\, :=\, {\mathcal L}'\otimes p^*_2 \xi^*
\, \longrightarrow\, X\times {\rm Pic}^{0}(X)\, ,
\end{equation}
where ${\mathcal L}'$ and $\xi$ are constructed in \eqref{e0}
and \eqref{e1} respectively. It is clear that ${\mathcal L}$
is real or quaternionic, and also it is topologically
trivial in the direction of ${\rm Pic}^{0}(X)$.

Let
\begin{equation}\label{f5}
A\, :=\, {\rm Pic}^{0}(X)^\vee
\end{equation}
be the moduli space of topologically trivial line bundles on
${\rm Pic}^{0}(X)$. The anti--holomorphic involution
$\widetilde{\sigma}$ of ${\rm Pic}^{0}(X)$ produces an
anti--holomorphic involution of $A$ by sending any
$\eta$ to $\widetilde{\sigma}^*\overline{\eta}$;
this anti--holomorphic involution of $A$ will be denoted
by $\sigma'$.

We have the morphism
\begin{equation}\label{e3}
f\, :\, X\, \longrightarrow\, A
\end{equation}
that sends any $x\, \in\, X$ to the point representing the
restriction ${\mathcal L}\vert_{\{x\}\times {\rm Pic}^{0}(X)}$,
where $\mathcal L$ is constructed in \eqref{e2}. Note that
$f$ is defined over $\mathbb R$, meaning $f(\sigma(x))\,=
\, \sigma'(f(x))$ for all $x\,\in\, X$.

Fix a point $x_0\, \in\, X$. The divisor
$$
D\, =\, \frac{g-1}{2}(x_0+\sigma(x_0))
$$
is real, meaning it is
fixed by $\sigma$. We have an isomorphism
$$
\phi\, :\, {\rm Pic}^{0}(X)\, \longrightarrow\,
{\rm Pic}^{g-1}(X)
$$
defined by $L\, \longmapsto\, L\bigotimes {\mathcal O}_X(D)$.
This morphism intertwines the anti--holomorphic involutions of
${\rm Pic}^{0}(X)$ and ${\rm Pic}^{g-1}(X)$, so it is defined
over $\mathbb R$. Let
\begin{equation}\label{e4}
\Theta_0\,:=\, \phi^{-1}(\Theta) \, \subset\, {\rm Pic}^{0}(X)
\end{equation}
be the inverse image of the theta divisor on ${\rm 
Pic}^{g-1}(X)$ (see the proof of Lemma \ref{lem1} for
the definition of theta divisor).

For any $L\, \in\, {\rm Pic}^{0}(X)$, let
$$
\tau_L\, :\, {\rm Pic}^{0}(X)\, \longrightarrow\, {\rm Pic}^{0}(X)
$$
be the isomorphism defined by $L'\, \longmapsto\,L'\bigotimes L$.
The divisor $\Theta_0$ in \eqref{e4} produces an isomorphism
\begin{equation}\label{e5}
{\rm Pic}^{0}(X)\, \stackrel{\sim}{\longrightarrow}\, A
\end{equation}
(see \eqref{f5} for $A$) by sending any $L$ to the line bundle
${\mathcal O}_{{\rm Pic}^{0}(X)}(\tau^*_L\Theta_0-\Theta_0)$;
this map is an isomorphism because $\Theta_0$ gives
a principal polarization on $\text{Pic}^0(X)$. Let
$$
\Theta'\, \subset\, A
$$
be the image of $\Theta_0$ by the isomorphism in \eqref{e5}.

Note that the pull back $f^*{\mathcal O}_A(\Theta')\, \longrightarrow
\,X$ is a real line bundle, where $f$ is the morphism
in \eqref{e3}. We have
\begin{equation}\label{b4}
\text{degree}(f^*{\mathcal O}_A(\Theta'))\, =\,g
\end{equation}
(see \cite[p. 336]{GH}).
We note that \eqref{b4} also follows from
\eqref{b0} and \eqref{z1}, because they imply that
$\text{degree}(\phi^*{\mathcal O}_{{\rm 
Pic}^{g-1}(X)}(\Theta)\vert_{X\times \{\zeta\}})\,=\,
\text{degree}(K_X)- \text{degree} (\zeta)\,=\, g$.

On the other hand, since $g$ is odd, there is
no real line bundle on $X$ of degree $g$
(recall that $\sigma$ does not have any fixed point).
This completes the proof of the proposition.
\end{proof}

\section{Real universal bundles of higher rank}\label{sec3}

\subsection{Existence of universal bundles}
Henceforth, we will assume that
\begin{enumerate}
\item the rank $r$ is at least two, and

\item $d$ is coprime to $r$.
\end{enumerate}

Recall the construction of the universal bundle in Lemma
\ref{lem4}.

\begin{remark}\label{rem2}
{\rm The integer $a$ in \eqref{u1} can be taken to be even
if and only if the integer $\chi$ in \eqref{chi}
is odd. Indeed, if $\chi$ is even,
then $r$ is odd (recall that $d$ is coprime to r), and
hence from \eqref{u1} we know that $ar$ is odd. Therefore, $a$
is odd if $\chi$ is even. If $\chi$ is odd, then $\chi$ is coprime
to $2r$. Hence there are integers $a'$ and $b$ such that
$2ra' +\chi b\,=\, -1$. Hence $a$ can be taken to be even
if $\chi$ is odd.}
\end{remark}

\begin{proposition}\label{prop-existence-2}
If the integer $\chi$ in \eqref{chi} is odd, then there is a
real universal
vector bundle over $X\times {\mathcal M}_X(r,d)$. 

Thus there are real universal bundles when:
\begin{itemize}
\item{}$g$ even: the degree and the rank are of opposite parity.
\item{}$g$ odd: the degree is odd.
\end{itemize}

Given a real (respectively, quaternionic) line bundle on $X$ of degree 
$d_0$ then 
there is a universal real bundle over $X\times {\mathcal M}_X(r,d)$ if 
and only if there is a universal real (respectively, quaternionic) bundle 
over $X\times {\mathcal M}_X(r,d + rd_0)$. 

In particular, for $d$ odd, there is a universal quaternionic vector
bundle over $X\times {\mathcal M}_X(r,d)$.
\end{proposition}

\begin{proof} In view of Remark \ref{rem2}, the first part is a
consequence
of Lemmas \ref{lem4}. The second part follows as in Lemma \ref{lem2}. For 
the 
third, let us consider first the case when the genus is odd. There is a 
quaternionic line bundle of degree zero on $X$. The second part of the 
proposition then tells us that there is universal quaternionic bundle if 
and
only if there is a real one in a given rank and degree; thus the 
quaternionic bundle 
exists for odd degree. For even genus, one has a quaternionic line 
bundle of degree one. One then has a universal quaternionic bundle over 
$X\times {\mathcal M}_X(r,d )$ if one has a universal real bundle over
$X\times {\mathcal M}_X(r,d-r )$. This happens when $r,d-r $ are of 
opposite parity, that is, when $d$ is odd.
\end{proof}

\subsection{Non existence; the real case of even degree} Assume that
the degree $d$ is even.

\begin{proposition}\label{prop3}
There is a real universal vector bundle over
$X\times {\mathcal M}_X(r,d)$ if and only if $\chi = d-r(g-1)$ is odd, 
i.e., if and only if $r$ is odd and $g$ is even.
\end{proposition}

\begin{proof}
First note that the rank $r$ is odd because $d$ is coprime
to $r$. So, if
$g$ is even, then $\chi$ is odd. Hence in that
case from Proposition \ref{prop-existence-2} it follows that
there is a real universal
vector bundle over $X\times {\mathcal M}_X(r,d)$.

Assume that $g$ is odd.

We will construct a real cyclic \'etale cover
\begin{equation}\label{u3}
f\, :\, \widetilde{X}\, \longrightarrow\, X
\end{equation}
of degree $r$. For that, first note that there is a real
line bundle over $X$ of order exactly $r$. Indeed, if
$\widetilde{\sigma}$ is the anti--holomorphic involution
of $\text{Pic}^0(X)$ defined by $\eta\, \longmapsto\,
\sigma^*\overline{\eta}$, then the connected component of
the fixed point locus $\text{Pic}^0(X)^{\widetilde{\sigma}}$
containing the trivial line bundle ${\mathcal O}_X$ is a
divisible group. Also, all line bundles lying in this
component are real because ${\mathcal O}_X$ is a real
line bundle. Take any point $L_0$ of order exactly $r$
of this component. So $L_0$ is a real line bundle of
order exactly $r$. Now consider the map of total spaces
of line bundles
$$
\rho\,:\, L_0\, \longrightarrow\, L^{\otimes r}_0
\,=\, {\mathcal O}_X
$$
defined by $v\,\longmapsto\, v^{\otimes r}$. Take
$$
\widetilde{X}\, =\, \rho^{-1}(1_X)\, ,
$$
where $1_X$ is the image of the section of ${\mathcal O}_X$ 
defined by the constant function $1$. The projection $f$
in \eqref{u3} is the restriction of the natural projection
of $L_0$ to $X$.

The anti--holomorphic involution of the total space of the
line bundle $L_0$ defining its real structure preserves
the subset $\widetilde{X}$. Therefore,
this restriction produces an anti--holomorphic
involution $\sigma_{\widetilde{X}}$ of $\widetilde{X}$. Note that
the projection $f$ intertwines $\sigma_{\widetilde{X}}$ and
$\sigma$.

We have
$$
\widetilde{g}\, :=\, \text{genus}(\widetilde{X})
\,=\, r(g-1)+1\, ,
$$
and $\widetilde{g}$ is odd because $g$ is odd.

Since $d$ is even and $\widetilde{g}$ is odd,
there is a quaternionic line bundle
\begin{equation}\label{x2}
\xi\, \longrightarrow\, \widetilde{X}
\end{equation}
of degree $d$ \cite[Theorem 6.6]{BHH} (see also
\cite[p. 55, Theorem 2.6]{AB}). Consider the direct
image
\begin{equation}\label{x3}
E\, :=\, f_*\xi\, \in\, {\mathcal M}_X(r,d)\, ,
\end{equation}
where $f$ is the projection in \eqref{u3};
since $f$ is unramified, the vector bundle
$f_*\xi$ is semistable
of degree $d$, hence it is stable (recall that
$d$ is coprime to $r$).

Since
$\xi$ is quaternionic, it follows that $E$ is
quaternionic; the direct image of the isomorphism
$$
\xi\,\longrightarrow\,\sigma^*_{\widetilde{X}}
\overline{\xi}
$$
defining the quaternionic structure of $\xi$ is
clearly a quaternionic structure on $E$.
Note that $E$ is a real point of ${\mathcal M}_X(r,d)$
meaning it is fixed by the anti--holomorphic involution
of ${\mathcal M}_X(r,d)$.

If there is a real universal vector bundle $\mathcal E$
over $X\times {\mathcal M}_X(r,d)$,
then the restriction of $\mathcal E$
to $X\times\{E\}\,=\, X$ is also real.
Consequently, in that case, $E$
would be both real and quaternionic. But that's impossible
\cite{BHH}. Hence there is no universal real
algebraic vector bundle over $X\times {\mathcal M}_X(r,d)$.
This completes the proof of the proposition.
\end{proof}

\subsection{Non-existence; the real case of odd degree.}

We now assume that $d$ is odd.

\begin{proposition}\label{prop4}
There is a real universal
vector bundle over $X\times {\mathcal M}_X(r,d)$ if and only if $\chi
\,= \,d-r(g-1)$ is odd, i.e., if and only if
either $r$ is even or $g$ is odd.

\end{proposition}

\begin{proof}
The cases for which the bundles must exist have been covered above; one 
must show non-existence in the case when $r$ is odd, $g$ is even.

Fix a real cyclic \'etale cover
\begin{equation}\label{u5}
f\, :\, \widetilde{X}\, \longrightarrow\, X
\end{equation}
of degree $r$ (see \eqref{u3} for its construction). Note that
\begin{equation}\label{ge}
\widetilde{g}\, :=\, \text{genus}(\widetilde{X})\, =\, 1+(g-1)r
\end{equation}
is even.

Let
$$
\xi\, \longrightarrow\, \widetilde{X}
$$
be a real line bundle of degree $d-r$ (note that
$d-r$ is even). The vector bundle
$$
E\, :=\, f_*\xi\, \longrightarrow\, X
$$
is stable of rank $r$ and degree $d-r$. The real structure
of $\xi$ defines a real structure on $E$. Let
\begin{equation}\label{u6}
\varphi\, :\, \text{Pic}^1(X)\, \longrightarrow\,
{\mathcal M}_X(r,d)
\end{equation}
be the morphism defined by $L\, \longmapsto\,
E\bigotimes L$. Since the vector bundle $E$ is real,
the morphism $\varphi$ is defined over $\mathbb R$, meaning
$\varphi$ intertwines the anti--holomorphic involutions
of $\text{Pic}^1(X)$ and ${\mathcal M}_X(r,d)$.

To prove by contradiction, let
$$
{\mathcal U}\, \longrightarrow\, 
X\times {\mathcal M}_X(r,d)
$$
be a real universal vector bundle. Let
$$
{\mathcal P}\, :=\, (\text{Id}_X\times\varphi)^*{\mathcal U}
\, \longrightarrow\, X\times \text{Pic}^1(X)
$$
be the pullback, where $\varphi$ is the map
in \eqref{u6}. Let
\begin{equation}\label{u7}
{\mathcal L}\,:=\, (\bigwedge\nolimits^r {\mathcal P})
\bigotimes p^*_X \bigwedge\nolimits^r E^*\,\longrightarrow\,
X\times \text{Pic}^1(X)
\end{equation}
be the line bundle, where $p_X$ is the projection of
$X\times \text{Pic}^1(X)$ to $X$. Since both the
vector bundles $\mathcal P$ and $E$ are real, it follows
that the line bundle $\mathcal L$ is also real.

There is a quaternionic universal line bundle
$$
{\mathcal L}_0\,\longrightarrow\,
X\times \text{Pic}^1(X)
$$
(see Proposition \ref{prop1}). Since $r$ is odd, the line bundle
${\mathcal L}^{\otimes r}_0$ is also quaternionic. Note that
for any $L\, \in\, \text{Pic}^1(X)$, the restrictions
of both the line bundles $\mathcal L$ and ${\mathcal L}^{\otimes r}_0$
to $X\times\{L\}$ are isomorphic to $L^{\otimes r}$. Hence there
is a unique, up to an isomorphism, holomorphic line bundle
\begin{equation}\label{l1}
L_1\, \longrightarrow\, \text{Pic}^1(X)
\end{equation}
such that
$$
{\mathcal L}^{\otimes r}_0\,=\, {\mathcal L}\otimes p^*_2 L_1\, ,
$$
where $p_2$ is the projection of $X\times \text{Pic}^1(X)$ to
$\text{Pic}^1(X)$. Since $\mathcal L$ is real and
${\mathcal L}^{\otimes r}_0$ is quaternionic, we conclude that
$L_1$ is quaternionic.

Since $g$ is even, there is a quaternionic line bundle on $X$
of degree one; this follows from the facts that there is a quaternionic 
theta characteristic on $X$, \cite[pp. 61--62]{At2}, and there is a real 
line bundle on $X$ of degree two (see the proof of Proposition
\ref{prop1}). Hence $\text{Pic}^1(X)$ has a real point. 
Therefore,
there is no quaternionic line bundle on $\text{Pic}^1(X)$
\cite[p. 201, Proposition 2.7.4]{CP}. This contradicts the
existence of $L_1$ in \eqref{l1}. Therefore, there is no
real universal vector bundle over
$X\times {\mathcal M}_X(r,d)$. This completes the proof of
the proposition.
\end{proof}

\subsection{Non-existence; the quaternionic case}
\begin{proposition} \label{propquat} There is a universal quaternionic
vector bundle over $X\times {\mathcal M}_X(r,d)$ if and only if $d$ is 
odd.\end{proposition}
The proof proceeds by using the results for the real case, tensoring by a 
quaternionic line bundle of degree 0 (if $g$ is odd) or 1 (if
$g$ is even); universal real
vector bundles exist if and only the corresponding real ones exist.

\section{Fixed determinant moduli space}\label{sec4}

As before, assume that $r\, \geq\, 2$.

Let $\zeta\,\longrightarrow\, X$ be a holomorphic line
bundle of degree $d$. Let
$$
{\mathcal M}_X(r,\zeta)\, \subset\,{\mathcal M}_X(r,d)
$$
be the moduli space of stable vector bundles $E$ with
$\bigwedge^r E\,=\, \zeta$. There is a universal vector
bundle over $X\times {\mathcal M}_X(r,\zeta)$ if and only
if $d$ is coprime to $r$ \cite{Ra}.

As in Section \ref{sec3}, assume that
$d$ is coprime to $r$.

Let $\zeta$ be a fixed point of the anti--holomorphic
involution $\widetilde\sigma$ of $\text{Pic}^d(X)$
(see \eqref{p2}). So the line bundle $\zeta\,\longrightarrow
\, X$ is either real or quaternionic.
The subvariety ${\mathcal M}_X(r,\zeta)$
is preserved by the anti--holomorphic involution of
${\mathcal M}_X(r,d)$ (see \eqref{p3}). The induced
anti--holomorphic involution of ${\mathcal M}_X(r,\zeta)$
will be denoted by $\widehat \sigma$.

A \textit{real universal vector bundle} over
$X\times {\mathcal M}_X(r,\zeta)$ is a universal vector
bundle
$$
{\mathcal E}\, \longrightarrow\,X\times {\mathcal M}_X(r,\zeta)
$$
equipped with a holomorphic isomorphism
$$
\sigma'\,:\, {\mathcal E}\,\longrightarrow\,
(\sigma\times\widehat{\sigma})^*\overline{\mathcal E}
$$
such that $\sigma'\circ\sigma'\,=\,\text{Id}_{\mathcal E}$.

Therefore, the restriction of a real universal vector bundle
over $X\times {\mathcal M}_X(r,d)$ to
$X\times {\mathcal M}_X(r,\zeta)$ is also real universal.

\begin{proposition}\label{prop5}
Assume that the line bundle $\zeta$ is real. There is a real
universal vector bundle over $X\times {\mathcal M}_X(r,\zeta)$ if
and only if $\chi(E)$ is odd for $E\,\in\, {\mathcal M}_X(r,\zeta)$.
\end{proposition}

\begin{proof}
Since $\zeta$ is real, the degree $d$ is even. So $r$ is
odd. In view of Proposition \ref{prop3}, the only case to check is
where $r$ and $g$ are both odd.

We can choose the line bundle $\xi\, \longrightarrow\,
\widetilde{X}$ in \eqref{x2} such that the vector bundle $f_*\xi$
in \eqref{x3} lies in ${\mathcal M}_X(r,\zeta)$. Indeed, this
follows from the fact that the connected component
of the fixed point locus $\text{Pic}^0(X)^{\widetilde{\sigma}}$
containing ${\mathcal O}_X$ is a divisible group.
Hence the
argument in Proposition \ref{prop3} completes of the proof
for the case where both $r$ and $g$ are odd.
\end{proof}

\begin{proposition}\label{prop6}
Assume that the line bundle $\zeta$ is quaternionic.
There is a real universal
vector bundle over $X\times {\mathcal M}_X(r,\zeta)$
if and only if there is a real universal
vector bundle over $X\times {\mathcal M}_X(r,d)$.
\end{proposition}

\begin{proof}
First assume that $d$ is even. As in Proposition \ref{prop5},
the only case to check is the one where both $r$ and $g$ are odd.
Again the argument in Proposition \ref{prop3} settles this case.

Now assume that $d$ is odd.

In view of Proposition \ref{prop4}, the only case to check is the one 
where $r$ is odd and $g$ is even.

Consider the covering $\widetilde X$ in \eqref{u5}.
Its genus $\widetilde{g}$ is even (see \eqref{ge}). Hence
there is a quaternionic line bundle
$$
L\, \longrightarrow\, \widetilde{X}
$$
such that $f_*L\, \in\, {\mathcal M}_X(r,\zeta)$, where $f$ and
$\widetilde{X}$ are as in \eqref{u5}. See the proof of Proposition
\ref{prop1} for the existence of $L$; we note that the determinant
can always be arranged to be $\zeta$ because
the connected component of $\text{Pic}^0(X)^{\widetilde{\sigma}}$
containing ${\mathcal O}_X$ is a divisible group.

The vector bundle $f_*L$ is
quaternionic because $L$ is quaternionic. In particular, the
point of ${\mathcal M}_X(r,\zeta)$ representing $f_*L$ is fixed
by the anti--holomorphic involution of ${\mathcal M}_X(r,\zeta)$.

If
$$
{\mathcal U}\, \longrightarrow\,
X\times {\mathcal M}_X(r,\zeta)
$$
is a real universal vector bundle, then the restriction
of ${\mathcal U}$ to $X\times\{f_*L\}$ is real. But we have
seen that
$f_*L$ is quaternionic. This is a contradiction \cite{BHH}.
Hence there is no real universal vector bundle
over $X\times {\mathcal M}_X(r,\zeta)$ if $r$ is odd and $g$
is even. This completes the proof of the proposition.
\end{proof}



\begin{thebibliography}{AAAA}

\bibitem[ACGH]{ACGH} E. Arbarello, M. Cornalba, P. A. Griffiths and
J. Harris: \textit{Geometry of algebraic curves.} Vol. I,
Springer-Verlag, New York, 1985. 

\bibitem[AB]{AB} P. Ar\'es-Gastesi and I. Biswas: The Jacobian 
of a nonorientable Klein surface, II, \textit{Proc. Indian Acad.
Sci. (Math. Sci.)} \textbf{113} (2003), 139--152.

\bibitem[At1]{At1} M. F. Atiyah: Vector bundles over an elliptic
curve, \textit{Proc. London Math. Soc.} \textbf{7} (1957), 412--452.

\bibitem[At2]{At2} M. F. Atiyah: Riemann surfaces and spin
structures, \textit{Ann. Sci. \'Ec. Norm. Sup.} \textbf{4}
(1971), 47--62.

\bibitem[BM]{BM} A. Beilinson and Y. I. Manin: The Mumford
form and the Polyakov measure in string theory, \textit{Comm.
Math. Phys.} \textbf{107} (1986), 359--376.

\bibitem[BHH]{BHH} I. Biswas, J. Huisman and J. Hurtubise: The
moduli space of stable vector bundles over a real algebraic curve,
\textit{Math. Ann.} \textbf{347} (2010), 201--233.

\bibitem[CP]{CP} C. Ciliberto and C. Pedrini: Real abelian varieties
and real algebraic curves, in: \textit{Lectures in real geometry}
(Madrid, 1994), 167--256, de Gruyter Exp. Math., 23, de Gruyter,
Berlin, 1996.

\bibitem[De]{De} P. Deligne: Le d\'eterminant de la cohomologie,
\textit{Current trends in arithmetical algebraic 
geometry} (Arcata, Calif., 1985), 93--177,
Contemp. Math., 67, Amer. Math. Soc., Providence, RI, 1987. 

\bibitem[GH]{GH} P. A. Griffiths and J. Harris:
\textit{Principles of algebraic geometry}, Pure and Applied
Mathematics, Wiley-Interscience, New York, 1978.

\bibitem[Gr]{Gr} A. Grothendieck: Sur la classification
des fibr\'es holomorphes sur la sph\`ere de Riemann,
\textit{Amer. Jour. Math.} \textbf{79} (1957), 121--138.

\bibitem[HJ]{HJ} N.-K. Ho and L. C. Jeffrey, The volume of the 
moduli space of flat connections on a nonorientable 2-manifold,
\textit{Comm. Math. Phys.} \textbf{256} (2005), 539--564.

\bibitem[HL]{HL} N.-K. Ho and C.-C. M. Liu: Yang-Mills 
connections on nonorientable surfaces, \textit{Comm. Anal. Geom.}
\textbf{16} (2008), 617--679.

\bibitem[JK]{JK}  L. C. Jeffrey and F. C. Kirwan: Intersection
theory on moduli spaces of holomorphic bundles of arbitrary rank on a Riemann
surface, \textit{Ann. Math.} \textbf{148} (1998), 109--196.

\bibitem[Mu]{Mu} D. Mumford: \textit{Geometric invariant theory},
Springer-Verlag, 1965.

\bibitem[MN]{MN} D. Mumford and P. E. Newstead: Periods of
a moduli space of bundles on curves, \textit{Amer. Jour. Math.}
\textbf{90} (1968), 1200--1208.

\bibitem[Ne]{Ne} P. E. Newstead: A non--existence theorem for
families of stable bundles, \textit{Jour. London Math. Soc.}
\textbf{6} (1973), 259--266.

\bibitem[Ra]{Ra} S. Ramanan: The moduli spaces of vector bundles
over an algebraic curve, \textit{Math. Ann.} \textbf{200} (1973),
69--84.

\end{thebibliography}
\end{document}